\documentclass[leqno]{amsart}

\usepackage{amsmath,amsthm,amsfonts,latexsym,amssymb,mathrsfs}

\newtheorem{thm}{Theorem}

\newtheorem{lem}[thm]{Lemma}

\newtheorem{cor}[thm]{Corollary}
\numberwithin{thm}{section}
\numberwithin{equation}{section}

\theoremstyle{definition}

\newcommand{\rat}{\mathbb Q}

\newcommand{\alg}{\overline\rat}

\newcommand{\intg}{\mathbb Z}

\newcommand{\mult}{\lambda}
\newcommand{\ord}{\mult}

\title[Weil height and auxiliary polynomials]{The Weil height in terms of
an auxiliary polynomial}
\author[C.L. Samuels]{Charles L. Samuels}
\address{Department of Mathematics, University of Texas at Austin, 1 University Station C1200
  Austin, TX 78712}
\email{csamuels@math.utexas.edu}

\begin{document}

\begin{abstract}
  Recent theorems of Dubickas and Mossinghoff use auxiliary polynomials to give lower bounds on the
  Weil height of an algebraic number $\alpha$ under certain assumptions on $\alpha$.  We prove a
  theorem which introduces an auxiliary polynomial for giving lower bounds on the height of any algebraic
  number.  Our theorem contains, as corollaries, a slight generalization of the above results as well as 
  some new lower bounds in other special cases.
\end{abstract}

\maketitle

\section{Introduction}

Let $K$ be a number field and $v$ a place of $K$ dividing the place $p$ of $\rat$.  Let $K_v$ and $\rat_p$
denote the respective completions.  We write $\|\cdot\|_v$ to denote the unique absolute value on $K_v$ 
extending the $p$-adic absolute value on $\rat_p$ and let  $|\cdot|_v=\|\cdot\|_v^{[K_v:\rat_p]/[K:\rat]}$.
Define the logarithmic {\it Weil height} of $\alpha\in K$ by
\begin{equation*} \label{WeilHeightDef}
  h(\alpha) = \sum_v\log^+|\alpha|_v
\end{equation*}
where the sum is taken over all places $v$ of $K$.  By the way we have normalized our absolute values,
this definition does not depend on $K$, and therefore, $h$ is a well-defined function on $\alg$.
By Kronecker's Theorem, $h(\alpha)\geq 0$ with equality precisely when $\alpha$ is zero or a root of unity.

For $f\in\intg[x]$ having roots $\alpha_1,\ldots,\alpha_d$ define the logarithmic {\it Mahler measure} of $f$ by 
\begin{equation*} \label{MahlerMeasureDef}
  \mu(f) = \sum_{k=1}^d h(\alpha_k).
\end{equation*}
It is also worth noting that if $f$ is irreducible then $\mu(f) = \deg\alpha\cdot h(\alpha)$.

Certainly $\mu(f)\geq 0$ with equality precisely when the only roots of $f$ are $0$ and roots of unity.  
In 1933, D.H. Lehmer \cite{Lehmer} asked if there is a constant $c>0$ such that $\mu(f)\geq c$ 
in all other cases.  He noted that 
\begin{equation*}
  \mu(x^{10}+x^9-x^7-x^6-x^5-x^4-x^3+x+1) = .1623\ldots
\end{equation*}
and this remains the smallest known Mahler measure greater than $0$.
The best known unconditional result toward answering Lehmer's problem is a theorem of Dobrowolski 
\cite{Dobrowolski} where he proves that if $f$ has positive Mahler measure then
\begin{equation*}
  \mu(f) \gg \left(\frac{\log\log \deg f}{\log \deg f}\right)^3.
\end{equation*}

An affirmative answer to Lehmer's problem has been given in certain special cases.
A polynomial $f$ is said to be reciprocal if whenever $\alpha$ is a root of $f$ then $\alpha^{-1}$ is
also a root.  Breusch \cite{Breusch} proved that there exists a positive constant $c$ such that
if $f$ is not reciprocal then $\mu(f)\geq c$.
Smyth \cite{Smyth} later showed that we may take $c = \mu(x^3-x+1)$.  
Borwein, Hare and Mossinghoff \cite{BHM} improved the constant
found by Smyth in the special case that $f$ has odd coefficients.  They showed that
if $f$ is a non-reciprocal polynomial over $\intg$ having odd coefficients, then 
$\mu(f) \geq \mu(x^2-x-1)$.

Borwein, Dobrowolski and Mossinghoff \cite{BDM} relaxed the assumption that $f$ not be reciprocal and 
still obtained an absolute lower bound on $\mu(f)$.  They used properties of the resultant to prove
that if $f$ has no cyclotomic factors and coefficients congruent to $1\mod m$ then
\begin{equation*} \label{BDMInequality}
  \mu(f) \geq c_m\cdot\frac{\deg f}{1+\deg f}
\end{equation*}
where $c_2=(\log 5)/4$ and $c_m = \log(\sqrt{m^2+1}/2)$ for all $m>2$.  These results appear in \cite{BDM} 
as Corollaries 3.4 and 3.5 to Theorem 3.3.  This theorem gives a lower bound of the form
\begin{equation} \label{BDMMain}
  \mu(f) \geq c_m(T)\cdot\frac{\deg f}{1+\deg f}
\end{equation}
where $f$ has no cyclotomic factors and coefficients congruent to $1\mod m$.  Here, $c_m(T)$ is a positive 
constant depending on  both $m$ and an auxiliary polynomial $T\in\intg[x]$.  The corollaries follow by making an 
appropriate choice of $T$.

Extending the techniques of \cite{BDM}, Dubickas and Mossinghoff \cite{DubMoss} improved inequality 
\eqref{BDMMain} by finding a lower bound of the form
\begin{equation} \label{DubMossMain}
  \mu(g) \geq b_m(T)\cdot\frac{\deg g}{1+\deg f}
\end{equation}
where $b_m(T)\geq c_m(T)$.  Here, $g$ has no cyclotomic factors and is a factor of a polynomial $f$ having 
coefficients congruent to $1\mod m$.  
Moreover, they produced an algorithm which generates a sequence of polynomials $\{T_k\}$ 
such that the sequence $\{b_m(T_k)\}$ is increasing and $b_m(T_k) > c_m$ for sufficiently large $k$.

In a slightly different direction, Schinzel \cite{Schinzel} proved that if $\alpha$ is a totally real 
algebraic integer, not $0$ or $\pm 1$, then $h(\alpha) \geq \frac{1}{2}\log\frac{1+\sqrt 5}{2}$.
Bombieri and Zannier \cite{BomZan} proved that
if $\alpha$ is a totally $p$-adic algebraic number, not $0$ or a root of unity then
$h(\alpha) \geq \frac{\log p}{2(p+1)}$.

If, in addition, $\alpha$ is an algebraic unit, Petsche \cite{Petsche} gave the improved lower bound
\begin{equation} \label{ClayBound}
  h(\alpha) \geq \frac{c_p}{p-1}
\end{equation}
where $c_2 = \log(\sqrt 2)$ and $c_p = \log (p/2)$ for all primes $p>2$.  Dubickas and Mossinghoff
\cite{DubMoss} introduced an auxiliary polynomial to this problem as well, giving the lower bound
\begin{equation} \label{DubMossMain2}
  h(\alpha) \geq \frac{b_p(T)}{p-1}
\end{equation} 
where $b_p(T)$ is the same as in \eqref{DubMossMain}.  They showed how to find a sequence 
of auxiliary polynomials that further improved \eqref{ClayBound}.

As we have remarked, the well-known lower bounds \eqref{BDMMain}, \eqref{DubMossMain} and \eqref{DubMossMain2} 
all rely on an auxiliary polynomial $T$.  However, each of these bounds requires an assumption on $\alpha$.
Our main result, Theorem \ref{GlobalBounds}, shows that if $\alpha\in\alg$ then $h(\alpha)$ equals a function
depending on an auxiliary polynomial.  In section \ref{x^n-1}, we show that this theorem naturally 
contains the results of \cite{DubMoss}.  Finally, in sections \ref{(x^n-1)^r} and \ref{LowSup} we deduce 2
other interesting consequences to our main result.

\section{Main Results} \label{AuxiliaryPoly}

Let $\Omega_v$ be the completion of an algebraic closure of $K_v$. 
We define the logarithmic {\it local supremum norm} of $T\in\Omega_v[x]$ on the unit circle by
\begin{equation*} \label{LocalSupNormDef}
   \nu_v(T) = \log \sup\{|T(z)|_v:z\in\Omega_v\ \mathrm{and}\ |z|_v = 1\}.
\end{equation*}
For $\alpha\in\Omega_v$ and $N\in\intg$ such that $\deg T\leq  N$ define
\begin{equation*}
  U_v(N,\alpha,T) = \inf\{\nu_v(T-f): f\in\Omega_v[x], f(\alpha) = 0\ \mathrm{and}\ \deg f\leq N\}.
\end{equation*}
We now obtain the following lemma which relates $U_v(N,\alpha,T)$ to more familiar functions.

\begin{lem} \label{LocalBounds}
  Let $N\in\intg$ and $\alpha\in\Omega_v$.  If $T\in\Omega_v[x]$ is such that $\deg T\leq N$ then
  \begin{align} \label{LocalBoundsEquations}
    U_v(N,\alpha,T) & = \log|T(\alpha)|_v + U_v(N,\alpha,1) \nonumber \\
    & =  \log|T(\alpha)|_v - N\log^+|\alpha|_v.
  \end{align}
\end{lem}
\begin{proof}
If $T(\alpha) = 0$ then all parts of equations \eqref{LocalBoundsEquations} equal $-\infty$, so we assume that
$T(\alpha)\ne 0$. Let us first verify the left hand equation.  For simplicity define the set
$$S_v(\alpha,N) = \{f\in\Omega_v[x]:f(\alpha) = 0\ \mathrm{and}\ \deg f\leq N\}.$$  It is clear that
\begin{align*}
  U_v(N,\alpha,T) & = \inf\{\nu_v(T(x)-f(x)):f\in S_v(\alpha,N)\} \\
  & = \inf\{\nu_v(T(x) - (T(x) - T(\alpha) +f(x))):f\in S_v(\alpha,N)\} \\
  & = \inf\{\nu_v(T(\alpha) - f(x)):f\in S_v(\alpha,N)\} \\
  & = \inf\{\nu_v(T(\alpha)(1 - f(x))):f\in S_v(\alpha,N)\}.
\end{align*}
Since $\nu_v$ is the logarithm of a norm, we may factor $T(\alpha)$ out of the infimum to see that
\begin{align*}
   U_v(N,\alpha,T) & = \log|T(\alpha)|_v +  \inf\{\nu_v(1 - f(x)):f\in S_v(\alpha,N)\} \\
   & = \log|T(\alpha)|_v + U_v(N,\alpha,1)
\end{align*}
which establishes the left hand equality.

In order to establish the right hand equality we must show that $U_v(N,\alpha,1) = -N\log^+|\alpha|_v$.
We first claim that if $N\in\intg$ then
\begin{equation} \label{MaxPrinciple}
  \log|F(\alpha)|_v \leq \nu_v(F) + N\log^+|\alpha|_v
\end{equation}
for all $F\in\Omega_v[x]$ with $\deg F\leq N$.  To see this, write $F(x) = \sum_{k=0}^{\deg F}a_kx^k$.
If $v$ is non-Archimedean then we have that
\begin{equation} \label{NonArchSup}
  \nu_v(F) =  \log \max\{|a_k|_v:0\leq k \leq \deg F\}
\end{equation}
and \eqref{MaxPrinciple} follows from the strong triangle inequality.  We now assume that $v$ is Archimedean.
If $|\alpha|_v \leq 1$ then the inequality follows from the maximum principle.  
If $|\alpha|_v > 1$ then we obtain that
\begin{equation*}
  \log|\alpha^{-\deg F}F(\alpha)|_v \leq \nu_v(x^{\deg F}F(x^{-1})) = \nu_v(F)
\end{equation*}
and \eqref{MaxPrinciple} follows.

Now suppose that $f\in S_v(\alpha,N)$.  Therefore, $\deg(1-f)\leq N$
and inequality \eqref{MaxPrinciple} implies that
\begin{equation*}
  0 = \log|1-f(\alpha)|_v \leq \nu_v(1-f) + N\log^+|\alpha|_v.
\end{equation*}
This inequality holds for all  polynomials $f\in S_v(\alpha,N)$ so that
the right hand side may be replaced by its infimum over all such $f$.  That is, we obtain
$0 \leq  U_v(N,\alpha,1) + N\log^+|\alpha|_v$ so we find that
\begin{equation}\label{halfway}
  U_v(N,\alpha,1) \geq -N\log^+|\alpha|_v.
\end{equation}

We will now establish the opposite direction of \eqref{halfway} by making specific choices for $f$
to give upper bounds on $U_v(N,\alpha,1)$.  By taking $f\equiv 0$ we see easily that
$U_v(N,\alpha,1) \leq 0$.  Similarly, by taking $f(x) = 1 - (x/\alpha)^N$ we obtain
\begin{equation*}
  U_v(N,\alpha,1) \leq \nu_v(x/\alpha)^N = -N\log|\alpha|_v.
\end{equation*}
Hence
\begin{equation}
  U_v(N,\alpha,1) \leq \min\{0,-N\log|\alpha|_v\} = -N\log^+|\alpha|_v.
\end{equation}
\end{proof}

If $\alpha\in K$ and $T\in K[x]$ are such that $T(\alpha)\ne 0$ then Lemma \ref{LocalBounds} implies that
$U_v(N,\alpha,T)=0$ for all but finitely many places $v$ of $K$.  Hence, in this situation we may define
\begin{equation*} \label{UDef}
  U(N,\alpha,T) = \sum_vU_v(N,\alpha,T)
\end{equation*}
where $v$ runs over the places of $K$.  We note that this definition does not depend on $K$ so that $U$
is a well-defined function on $\{(\alpha,T)\in\alg\times\alg[x]: T(\alpha)\ne 0\}$.  We are
now prepared to state and prove our main result.

\begin{thm} \label{GlobalBounds}
  Let $N\in\intg$ and $\alpha\in\alg$.  If $T\in\alg[x]$ is such that $\deg T\leq N$ and 
  $T(\alpha) \ne 0$ then
  \begin{equation*}
    U(N,\alpha,T) = U(N,\alpha,1) = -Nh(\alpha).
  \end{equation*}
\end{thm}
\begin{proof}
Assume that $K$ is a number field containing $\alpha$ and the coefficients of $T$ and $v$ is a place of $K$.  
We know that the absolute value $|\cdot|_v$ satisfies the product formula $\prod_v|\beta|_v = 1$
for all $\beta\in K^\times$.  Hence, summing the equation of Lemma \ref{LocalBounds} over all places
$v$ of $K$ we get that
\begin{equation} \label{FinalIneq}
  U(N,\alpha,T) = U(N,\alpha,1) = -Nh(\alpha)
\end{equation}
which establishes the theorem.
\end{proof}

\section{Polynomials near $x^n-1$} \label{x^n-1}

As we have remarked, Theorem \ref{GlobalBounds} naturally generalizes the
results of Dubickas and Mossinghoff in \cite{DubMoss}.  We will give a single result
that contains both their bound on the Mahler measure of a polynomial having coefficients congruent to 
$1 \mod m$ and their bound on the height of a totally $p$-adic algebraic unit.

Let us begin by reconstructing the situation of \cite{DubMoss}.  For an auxiliary polynomial 
$T\in\intg[x]$ and a positive integer $m$ define
\begin{equation} \label{IntegerOmega}
  \omega_m(T) = \log\gcd\left\{\frac{m^kT^{(k)}(1)}{k!}: 0\leq k\leq \deg T\right\}.
\end{equation}
Also assume that $f$ is a polynomial of degree $n-1$ with integer coefficients congruent to $1\mod m$.
The authors prove (Theorem 2.2 of \cite{DubMoss}) that if $g$ is a factor of $f$ over $\intg$
satisfying $\gcd(g(x),T(x^n)) = 1$ then
\begin{equation} \label{DubMossBound}
  \mu(g) \geq \frac{\omega_m(T) - \nu_\infty(T)}{\deg T}\left(\frac{\deg g}{n}\right).
\end{equation}
Later they prove (Theorem 4.2 of \cite{DubMoss}) that if $\alpha$ is a totally $p$-adic algebraic
unit then
\begin{equation} \label{DubMossPAdic}
  h(\alpha) \geq \frac{\omega_p(T) - \nu_\infty(T)}{(p-1)\deg T}.
\end{equation}
Our goal is to produce a generalization of \eqref{DubMossBound} where $T$ and $f$ are allowed to have algebraic
coefficients.   Our version also contains \eqref{DubMossPAdic} as a corollary.

Before we begin, we make one final trivial remark regarding the hypotheses of \cite{DubMoss}.
The assumption that $f$ have degree $n-1$ and coefficients congruent to $1\mod m$ is equivalent to
the assumption that $(x-1)f(x) \equiv x^n-1\mod m$.  Therefore, 
we can make a slightly stronger conclusion by hypthesizing instead that $f(x)\equiv x^n-1\mod m$
and bounding the Mahler measure of all factors $g$ of $f$.

We will require a version of $\omega_m(T)$ defined in \eqref{IntegerOmega} that allows $m$ to be a
general algebraic number and $T$ to have any algebraic coefficients.  If $K$ is a number field,
$m\in K$ and $T\in K[x]$ define
\begin{equation} \label{Omega}
  \omega_m(T) = -\sum_{v\nmid\infty}\log\max\left\{\left|\frac{m^kT^{(k)}(1)}{k!}\right|_v:
    0\leq k\leq \deg T \right\}
\end{equation}
where the sum is taken over places $v$ of $K$.
By the way we have normalized our absolute values, this definition does not depend on $K$.  
Moreover, if $m\in\intg$ and $T\in\intg[x]$ then \eqref{Omega} is the same as the definition 
\eqref{IntegerOmega}.

If $\alpha, \beta, m\in K$, then we write $\alpha\equiv\beta\mod m$ if $|\alpha - \beta|_v \leq |m|_v$ 
for all $v\nmid\infty$.   Similarly, if $f,g\in K[x]$ we write 
$f\equiv g\mod m$ if $\nu_v(f-g) \leq \log |m|_v$ for all $v\nmid\infty$.  
Neither defintion depends on $K$ and both generalize the usual notions of congruence in $\intg$.
If $T\in K[x]$ we often write $\nu_\infty(T) = \sum_{v\mid\infty}\nu_v(T)$ where $v$ runs over places of $K$.
This notation again does not depend on $K$.

It will also be convenient for this section and future applications to define
$U_v(\alpha,T) = U_v(\deg T,\alpha,T)$ and $U(\alpha,T) = U(\deg T,\alpha,T)$.

Using the definitions above, we obtain our generalized version of the results of \cite{DubMoss}.

\begin{thm} \label{DubMossGen}
  Let $m$ be an algebraic number.  Suppose that $f\in \alg[x]$ has degree $n$ and
  $f(x) \equiv x^n-1\mod m$.  If $\alpha$ is a root of $f$ and $T\in \alg[x]$ is such that 
  $T(\alpha^n)\ne 0$ then
  \begin{equation*}
    h(\alpha) \geq \frac{\omega_m(T) - \nu_\infty(T)}{n\deg T}.
  \end{equation*}
\end{thm}
\begin{proof}
  Let $K$ be a number field containing $\alpha$ and the coefficients of $T$ and let $v$ index the places of $K$.
  Using Theorem \ref{GlobalBounds} with $N=\deg T$ and the definition of $U_v$ we have that
  \begin{equation} \label{InitialUpperBound}
    -n\deg T\cdot h(\alpha) \leq \sum_{v\nmid\infty}U_v(\alpha, T(x^n)) + \nu_\infty(T)
  \end{equation}
  so we must show that $\sum_{v\nmid\infty}U_v(\alpha, T(x^n)) \leq - \omega_m(T)$.  Let $v\nmid\infty$.
  Writing $T$ in its Taylor expansion at $1$ and using the binomial theorem we find that
  \begin{align*}
    U_v(\alpha, T(x^n)) & = U_v\left(\alpha, \sum_{k=0}^{\deg T}\frac{T^{(k)}(1)}{k!}(x^n-1)^k\right) \\
    & \leq \nu_v\left( \sum_{k=0}^{\deg T}\frac{T^{(k)}(1)}{k!}(x^n-1-f(x))^k\right).
  \end{align*}
  Then using the strong triangle inequality for $\nu_v$ we obtain
  \begin{equation*}
    U_v(\alpha, T(x^n)) \leq \max\left\{\log\left|\frac{T^{(k)}(1)}{k!}\right|_v+ k\nu_v(x^n-1-f(x)): 
      0\leq k\leq \deg T\right\}.
  \end{equation*}
  Since $f(x)\equiv x^n-1\mod m$ we have that $\nu_v(x^n-1-f(x)) \leq \log|m|_v$.
  Consequently, we obtain that
  \begin{equation*}
    \sum_{v\nmid\infty}U_v(\alpha, T(x^n)) \leq 
    \sum_{v\nmid\infty}\log\max\left\{\left|\frac{m^kT^{(k)}(1)}{k!}\right|_v: 
      0\leq k\leq \deg T\right\} = -\omega_m(T)
  \end{equation*}
  and the theorem follows from \eqref{InitialUpperBound}.
\end{proof}

If we assume that $f$ and $T$ have integer coefficients and $m$ is a positive integer then we recover 
Theorem 2.2 of \cite{DubMoss}.

\begin{cor} \label{DubMoss}
  Let $f\in\intg[x]$ have degree $n$ and $f(x)\equiv x^n-1\mod m$.  If $g$ is a factor of $f$
  and $T\in\intg[x]$ is such that $\gcd(g(x),T(x^n)) = 1$ then
  \begin{equation*}
    \mu(g) \geq \frac{\omega_m(T) - \nu_\infty(T)}{\deg T}\left(\frac{\deg g}{n}\right).
  \end{equation*}
\end{cor}
\begin{proof}
  Apply Theorem \ref{DubMossGen} to each root $\alpha$ of $g$ and the result follows.
\end{proof}

We also recover Theorem 4.2 of \cite{DubMoss} giving a lower bound on the height of a totally
$p$-adic algebraic unit.

\begin{cor} \label{PAdicDubMoss}
  If $\alpha$ is a totally $p$-adic algebraic unit and $T\in\intg[x]$ is such that
  $T(\alpha^{p-1}) \ne 0$ then
  \begin{equation*}
    h(\alpha) \geq \frac{\omega_p(T) - \nu_\infty(T)}{(p-1)\deg T}.
  \end{equation*}
\end{cor}
\begin{proof}
  For a general number field $K$ and a non-Archimedean place $v$ of $K$ dividing the place $p$ of $\rat$, 
  let $O_v = \{x\in K_v:|x|_v\leq 1\}$ denote the ring of $v$-adic integers in $K_v$ and let $\pi_v$ 
  be a generator of its unique maximal ideal $M_v= \{x\in K_v:|x|_v< 1\}$.  Let 
  $d_v = [K_v:\rat_p]$ denote the local degree and $d = [K:\rat]$ the global degree.
  We also define the residue degree $f_v$ by $p^{f_v} = |O_v/M_v|$ and note that $|\pi_v|_v = \|p\|_v^{f_v/d}$.
  If $K$ is a totally $p$-adic field then we have that $f_v = d_v = 1$ for all $v\mid p$.

  Now assume that $K$ is the totally $p$-adic field $\rat(\alpha)$.  If $v$ is a place of $K$ dividing $p$ then
  \begin{equation*}
    |\alpha^{p-1}-1|_v \leq |\pi_v|_v = \|p\|_v^{f_v/d} = \|p\|_v^{d_v/d} = |p|_v
  \end{equation*}
  and if $v$ does not divide $p$ or $\infty$ then
  \begin{equation*}
    |\alpha^{p-1}-1|_v \leq 1 = |p|_v.
  \end{equation*}
  Hence we have that $x^{p-1}-1 \equiv x^{p-1} - \alpha^{p-1}\mod p$.  
  Now we may apply Theorem \ref{DubMossGen} with $m=p$ and $f(x) = x^{p-1} - \alpha^{p-1}$ and the result follows.
\end{proof}

\section{Polynomials near $(x^n-1)^r$} \label{(x^n-1)^r}

In this section, we apply Theorem \ref{GlobalBounds} in 
order to examine the Mahler measure of any factor of a polynomial $f$ satsifying 
$f(x) \equiv (x^n-1)^r \mod m$.  In particular, we obtain the following explicit lower bound.

\begin{thm} \label{NotDistinctCyclos}
  Suppose that $f\in\intg [x]$ has degree $nr$, $m\geq 2$ is an integer, and $f(x)\equiv (x^n-1)^r\mod m$.  
  If $g$ is a factor of $f$ over $\intg$ having no cyclotomic factors then
  \begin{equation*} \label{NotDistinctCyclosInequality}
    \mu(g) \geq c\cdot\left(\frac{\deg g}{n2^r}\right)
  \end{equation*}
  where $c$ is the unique positive real number satisfying 
  $\displaystyle ce^{c/2}\log 3 = \log(3/2)\log 2$.\linebreak {\rm (Note that $c=.22823\ldots$)}.
\end{thm}

As an application, let $T$ be a product of cyclotomic polynomials of degree $2N$.  
Then we may apply Theorem \ref{NotDistinctCyclos} with $g(x) = T(x) + mx^N$ where $|m|\geq 2$.  
In this situation, $r$ is the maximum multiplicity of the cyclotomic polynomials in the factorization of 
$T$ over $\intg$.  These types of polynomials have been studied extensively (see, for example, \cite{MPV})
and our results yield a lower bound on any such $g$, although it is not absolute for this
entire class of polynomials.   

Of course, Theorem \ref{NotDistinctCyclos} is not helpful when 
$g$ is a product of cyclotomic polynomials with the middle coefficient shifted by only $1$.
Numerical evidence presented in \cite{MPV} suggests that these polynomials form a relatively rich 
collection of polynomials of small Mahler measure.  Hence it would be useful to have a method for
giving lower bound on their Mahler measure.  However, we are unable to do so in this paper.

We also note that Theorem \ref{NotDistinctCyclos} is weaker than Corollaries 3.3 and 3.4 of \cite{BDM}
when $r=1$.  In this situation, we may appeal to \cite{DubMoss} or the results section \ref{x^n-1} to obtain
the sharpest known bounds.

The proof of Theorem \ref{NotDistinctCyclos} will require 3 lemmas as well as some additional notation.
Suppose that $g$ and $T$ are polynomials over any field $K$.  $K[x]$ is certainly a unique factorization
domain so we may write $\mult_g(T)$ to denote the mulitplicity of $g$ in the
factorization of $T$.  If $G$ is a collection of polynomials over $K$, then let 
$\mult_G(T) = \sum_{g\in G}\mult_g(T)$.  

Our first lemma is a direct generalization of Theorem 3.3 of \cite{BDM}.

\begin{lem} \label{Cyclos}
  Suppose that $f\in\intg [x]$ has degree $nr$ and $f(x)\equiv (x^n-1)^r\mod m$. If $g$ is a 
  factor of $f$ over $\intg$ and $T\in\rat [x]$ is relatively prime to $g$ then
  \begin{equation} \label{CyclosEquation1}
    \mu(g) \geq \frac{\ord_{x^n-1}(T)\log m-r\nu_\infty(T)}{r\deg T}\cdot\deg g.
  \end{equation}
  Moreover, if $2|m$ then
  \begin{equation} \label{CyclosEquation2}
    \mu(g) \geq \frac{\ord_{x^n-1}(T)\log m + \ord_{G_{n}}(T)\log 2 - r\nu_\infty(T)}{r\deg T}
    \cdot\deg g
  \end{equation}
  where $G_n = \{x^{n2^j}+1:j\geq 0\}$.
\end{lem}
\begin{proof}
  Suppose that $\alpha$ is a root of $f$, $K$ is a number field containing $\alpha$ and $v$
  indexes the places of $K$.  First observe that if $F_1,F_2\in \Omega_v[x]$ then 
  $\nu_v(F_1F_2)\leq \nu_v(F_1) + \nu_v(F_2)$.  This yields the multiplicativity relation
  \begin{equation} \label{USubMult}
    U_v(\alpha,F_1F_2)\leq U_v(\alpha,F_1) + U_v(\alpha,F_2).
  \end{equation}
    
  Theorem \ref{GlobalBounds} implies that
  \begin{equation} \label{Initial}
    -r\deg T\cdot h(\alpha) \leq \sum_{v\nmid\infty}U_v(\alpha,T^r) + r\nu_\infty(T).
  \end{equation}
  Suppose that that $T_0\in\intg[x]$ is such that $T(x)^r = (x^n-1)^{r\ord_{x^n-1}(T)}T_0(x)$.
  We know that since $T_0$ has integer coefficients, $U_v(\alpha,T_0) \leq \nu_v(T_0) \leq 0$.
  Then \eqref{USubMult} implies that
  \begin{align*}
    U_v(\alpha,T^r)  & \leq \ord_{x^n-1}(T)U_v(\alpha, (x^n-1)^r) \\
    & \leq \ord_{x^n-1}(T) \nu_v((x^n-1)^r -f(x)).
  \end{align*}
  Since $f$ has integer coefficients and satisfies $f(x)\equiv (x^n-1)^r\mod m$ we know that
  $\sum_{v\nmid\infty}\nu_v((x^n-1)^r -f(x)) \leq -\log m$.  It follows that
  \begin{equation} \label{AlmostFinal}
    -r\deg T\cdot h(\alpha) \leq -\ord_{x^n-1}(T)\log m + r\nu_\infty(T).
  \end{equation}
  Applying \eqref{AlmostFinal} to each root $\alpha$ of $g$, we obtain \eqref{CyclosEquation1}.

  Next, assume that $2|m$.  In this situation, write
  \begin{equation*}
    T(x)^r = T_0(x)(x^n-1)^{r\ord_{x^n-1}(T)}\prod_{j\geq 0}(x^{n2^j}+1)^{r\ord_{x^{n2^j}+1}(T)}
  \end{equation*}
  for some $T_0\in\intg[x]$. In addition to the congruence $f(x)\equiv (x^n-1)^r\mod m$,
  for each $j\geq 0$ there exists $b_j\in\intg[x]$ such that $f(x)b_j(x) \equiv (x^{n2^j}+1)^r \mod 2$.  
  Hence, it follows that
  \begin{equation*}
    \sum_{v\nmid\infty}\nu_v(x^{n2^j}+1 -f(x)b_j(x)) \leq -\log 2
  \end{equation*}
  for all $j\geq 0$.  Now we find that
  \begin{equation*}
    U_v(\alpha,T^r) \leq \ord_{x^n-1}(T) \nu_v((x^n-1)^r -f(x))
    + \sum_{j\geq 0} \ord_{x^{n2^j}+1}(T) \nu_v(x^{n2^j}+1 -f(x)b_j(x))
  \end{equation*}
  for all $v\nmid\infty$.  Therefore, \eqref{Initial} yields
  \begin{equation*}
    -r\deg T\cdot h(\alpha) \leq -\ord_{x^n-1}(T)\log m - \ord_{G_{n}}(T)\log 2 +r\nu_\infty(T)
  \end{equation*}
  and the result follows by a similar argument as above.
\end{proof}

Note that the right hand sides of the inequalities of Lemma \ref{Cyclos} are less than
$0$ when $r$ is too large compared to $m$.  Hence, it may appear that these bounds are useful
only when $r$ is small.  However, a simple consequence of Lemma \ref{Cyclos} allows us to give
non-trivial lower bounds when $r$ is large.

\begin{lem} \label{Cyclos2}
  Let $p$ be prime and $q$ a power of $p$ such that $\deg f = nq$ and \linebreak
  $f(x)\equiv (x^n-1)^q\mod p$.  If $g$ is a factor of $f$ over $\intg$ and $T\in\rat [x]$
  is such that $\gcd(T(x^q),g(x))=1$ then
  \begin{equation} \label{CyclosEquation3}
    \mu(g) \geq \frac{\ord_{x^{n}-1}(T)\log p - \nu_\infty(T)}{q\deg T}\cdot\deg g.
  \end{equation}
  Moreover, if $p=2$ then
  \begin{equation} \label{CyclosEquation4}
    \mu(g) \geq \frac{(\ord_{x^{n}-1}(T) + \ord_{G_{n}}(T))\log 2 - \nu_\infty(T)}{q\deg T}\cdot\deg g
  \end{equation}
  where $G_n = \{x^{n2^j}+1:j\geq 0\}$.
\end{lem}
\begin{proof}
  We know that $f(x)\equiv (x^n-1)^q\equiv x^{nq}-1\mod p$.  Therefore, we may apply Lemma \ref{Cyclos}
  with $m=p$, $r=1$ and $T(x^q)$ in place of $T(x)$.  We obtain that
  \begin{align*}
    \mu(g) & \geq \frac{\ord_{x^{nq}-1}(T(x^q))\log p - \nu_\infty(T(x^q))}{q\deg T}\cdot\deg g \\
    & = \frac{\ord_{x^{n}-1}(T)\log p - \nu_\infty(T)}{q\deg T}\cdot\deg g.
  \end{align*}
  Inequality \eqref{CyclosEquation4} follows from a similar argument.
\end{proof}

In the hypotheses of Lemma \ref{Cyclos} we are given $f(x)\equiv (x^n-1)^r\mod m$,
so we may also apply Lemma \ref{Cyclos2} with $p$ a prime dividing $m$ and $q=p^{\lceil\log_pr\rceil}$.
We know that $(x^n-1)^{q-r}f(x) \equiv (x^n-1)^q \mod p$ so that Lemma \ref{Cyclos2} still
applies to any factor $g$ of $f$.

As we have noted, this method allows us to deduce non-trivial lower bounds
on the Mahler measure even when $r$ is large.  There is the disadvantage that $q$ is potentially 
much larger than $r$, making the inequalities of Lemma \ref{Cyclos2} weaker than
those of Lemma \ref{Cyclos} in some cases.  Furthermore, if $m$ has many prime factors, $p$ 
will be significantly smaller than $m$, again making the inequalities of Lemma \ref{Cyclos2} 
weaker than those of Lemma \ref{Cyclos}.  

As a general rule, we will use Lemma \ref{Cyclos} 
when $r$ is small and Lemma \ref{Cyclos2} when $r$ is large to obtain the best universal results.
We see this strategy in the proof of our next lemma.

\begin{lem} \label{UniversalEffectiveBounds}
  Suppose that $f \in\intg [x]$ has degree $nr$ and  $f(x)\equiv (x^n-1)^r\mod m$.  
  If $g$ is a factor of $f$ over $\intg$ having no cyclotomic factors then
  \begin{equation} \label{Basic1}
    \mu(g) \geq \log\left(\frac{m}{2^r}\right)\left(\frac{\deg g}{nr}\right).
  \end{equation}
  If $p$ is a prime dividing $m$ then
  \begin{equation} \label{Basic2}
    \mu(g) \geq \frac{1}{p}\log\left(\frac{p}{2}\right)\left(\frac{\deg g}{nr}\right)
  \end{equation}
  and if $2$ divides $m$ then
  \begin{equation} \label{Basic3}
    \mu(g) \geq \frac{\log 2}{4}\left(\frac{\deg g}{nr}\right).
  \end{equation}
\end{lem}
\begin{proof}
 To prove \eqref{Basic1}, we apply Lemma \ref{Cyclos} with $T(x) = x^n-1$ and the inequality
 follows immediately.

 To prove \eqref{Basic2}, we let $p$ be a prime dividing $m$ and set $q=p^{\lceil\log_pr\rceil}$.
 Therefore $q$ is an integer greater than or equal to $r$ so  that $(x^n-1)^{q-r}f(x) \equiv (x^n-1)^q\mod p$.  
 Using $T(x)=x^n-1$ with inequality \eqref{CyclosEquation3} of Lemma \ref{Cyclos2} we find that
 \begin{equation*}
   \mu(g) \geq \log\left(\frac{p}{2}\right)\left(\frac{\deg g}{nq}\right).
 \end{equation*}
 But we also know that $q = p^{\lceil\log_pr\rceil} < p^{1+\log_pr} = pr$ so that
 \begin{equation*}
   \mu(g) \geq \log\left(\frac{p}{2}\right)\left(\frac{\deg g}{npr}\right)
 \end{equation*}
 which is the desired inequality.

 Finally, to prove \eqref{Basic3}, suppose that $2\mid m$ and $q=2^{\lceil\log_2r\rceil}$.
 Use $T(x) = x^{2n}-1$ in inequality \eqref{CyclosEquation4} of Lemma \ref{Cyclos2}
 to obtain the desired result.
\end{proof}

\noindent {\it Proof of Theorem \ref{NotDistinctCyclos}.}
Let $c_0 = c/(2\log 2)$.  We distinguish the following 3 cases.
\renewcommand{\theenumi}{{\it \roman{enumi}}}
\begin{enumerate}
\item\label{LargeM'} $m\geq 2^{r+c_0}$,
\item\label{SmallM'} $m < 2^{r+c_0}$ and $2\mid m$,
\item\label{MediumM'} $m < 2^{r+c_0}$ and $2\nmid m$.
\end{enumerate}
If $m\geq 2^{r+c_0}$ then we use inequality \eqref{Basic1} of Lemma \ref{UniversalEffectiveBounds} to find that
\begin{equation*} \label{LargeM}
  \mu(g) \geq c_0\log 2 \left(\frac{\deg g}{nr}\right)
  \geq 2c_0\log 2\left(\frac{\deg g}{n2^r}\right) = c\cdot\left(\frac{\deg g}{n2^r}\right).
\end{equation*}
If $m<2^{r+c_0}$ and $2\mid m$ then inequality \eqref{Basic3} implies that
\begin{equation*} \label{P=2}
  \mu(g) \geq \frac{\log 2}{4}\left(\frac{\deg g}{nr}\right) \geq 
  \frac{\log 2}{2}\left(\frac{\deg g}{n2^r}\right) \geq c\cdot\left(\frac{\deg g}{n2^r}\right).
\end{equation*}
If $m<2^{r+c_0}$ and $p\ne 2$ is a prime dividing $m$ then
we apply inequality \eqref{Basic2} to find that
\begin{align*}
  \mu(g) & \geq \frac{1}{p}\log\left(\frac{p}{2}\right)\left(\frac{\deg g}{nr}\right) \\
  & \geq \left(1-\frac{\log 2}{\log p}\right)\left(\frac{\log p}{p}\right)
  \left(\frac{\deg g}{nr}\right) \\
  & \geq \left(\frac{\log (3/2)}{\log 3}\right)\left(\frac{\log p}{p}\right)
  \left(\frac{\deg g}{nr}\right).
\end{align*}
However, the function $(\log x)/x$ is decreasing for $x\geq e$.  Since $p\leq m < 2^{r+c_0}$,
we conclude that
\begin{equation*}
  \frac{\log p}{p} > \frac{(r+c_0)\log 2}{2^{r+c_0}}
  > \frac{r\log 2}{2^{r+c_0}},
\end{equation*}
and hence,
\begin{equation*} \label{Pne2}
  \mu(g) \geq  \left(\frac{\log (3/2)\log 2}{2^{c_0}\log 3}\right)
  \left(\frac{\deg g}{n2^r}\right).
\end{equation*}
We know that $2^{c_0} = e^{c/2}$ so that by our definition of $c$ we obtain
\begin{equation*}
   \mu(g) \geq c\cdot\left(\frac{\deg g}{n2^r}\right)
\end{equation*}
which establishes the theorem in the final case.\qed

\section{Polynomials near polynomials of low Archimedean supremum norm} \label{LowSup}

Suppose that $m$ is a non-zero algebraic number.  We now examine the situation where $f$ and $T$ are polynomials 
over $\alg$ of the same degree with $f\equiv T\mod m$.  If $K$ is a number field containing $m$ with
$v$ indexing the places of $K$, let
\begin{equation*} \label{mNorm}
  N(m) = \sum_{v\mid\infty}\log|m|_v = -\sum_{v\nmid\infty}\log|m|_v.
\end{equation*}
Note that this definition does not depend on $K$ and the second equality follows from the
product formula.  Recall that we write $\nu_\infty(T) = \sum_{v\mid\infty}\nu_v(T)$ and we say that
$f\equiv T\mod m$ if $\nu_v(T-f) \leq \log |m|_v$ for all $v\nmid\infty$.

\begin{thm} \label{LowSupNorm}
  Suppose that $f$ and $T$ are polynomials over $\alg$ of the same degree such that $f\equiv T\mod m$.  
  If $\alpha$ satisfies $f(\alpha)=0$ and $T(\alpha)\ne 0$ then
  \begin{equation*}
    \deg T\cdot h(\alpha) \geq  N(m) - \nu_\infty(T).
  \end{equation*}
\end{thm}
\begin{proof}
  Let $K$ be a number field containing $\alpha$, $m$, the coefficients of $T$ and the coefficients of $f$.
  By Theorem \ref{GlobalBounds} we find that
  \begin{equation*}
    -\deg T\cdot h(\alpha) \leq \sum_{v\nmid\infty}U_v(\alpha,T) + \nu_\infty(T).
  \end{equation*}
  If $v\nmid\infty$ then $U_v(\alpha,T) \leq \nu_v(T-f) \leq \log |m|_v$ and the result follows.
\end{proof}

Clearly, in order for Theorem \ref{LowSupNorm} to yield a nontrivial lower bound, we must have 
that $N(m) > \nu_\infty(T)$, justifying the title of this section.
That is, if $f$ is sufficiently close to $T$ at enough non-Archimedean places of $K$, the
positive contribution from $N(m)$ will overcome the negative contribution
from $\nu_\infty(T)$.  We also note the special case of Theorem \ref{LowSupNorm} where $m\in\intg$
and $f,T\in\intg[x]$.

\begin{cor} \label{RatLowSupNorm}
  Suppose that $f$ and $T$ are polynomials over $\intg$ of the same degree and $m$ is a positive integer
  such that $f\equiv T\mod m$.  If $g$ is a factor of $f$ relatively prime to $T$ then
  \begin{equation*}
    \deg f\cdot\mu(g) \geq \deg g\cdot(\log m - \nu_\infty(T)).
  \end{equation*}
\end{cor}
\begin{proof}
  Apply Theorem \ref{LowSupNorm} to each root $\alpha$ of $g$ and the corollary follows.
\end{proof}

\begin{cor} \label{RatLowSupNorm2}
  Suppose that $f$ and $T$ are polynomials over $\intg$ of the same degree and $m$ is a positive integer
  such that $f\equiv T\mod m$.  If $f$ is relatively prime to $T$ then
  \begin{equation*}
    \mu(f) \geq \log m - \nu_\infty(T).
  \end{equation*}
\end{cor}
\begin{proof}
  Apply Corollary \ref{RatLowSupNorm} with $g=f$ and the result is immediate.
\end{proof}

\section{Acknowledgments}

The author wishes to thank J. Garza for noticing that the language of heights and places yields a 
more easily generalized proof of the results of \cite{BDM}.  We also thank F. Rodriguez-Villegas 
for remarking that the hypthesis that $f$ have coefficients congruent to $1$ mod $m$ may be replaced
by a hypothesis giving a more general congruence relation in $\intg$.
We thank M.J. Mossinghoff for his many useful suggestions, in particular, the inclusion of \cite{Breusch}
in the introduction.  Finally, we thank J.D. Vaaler for noticing that equality occurs in Theorem \ref{GlobalBounds}
for all auxiliary polynomials $T$ along with many other ideas.

\end{document}